\documentclass[preprint,12pt]{elsarticlehack}

\usepackage{amsmath,amsthm,amssymb}
\usepackage{graphicx}
\usepackage{hyperref}
\usepackage{mathrsfs}
\usepackage{color}
\usepackage{tikz}
\usepackage{tikzscale}
\usepackage{enumitem}
\usepackage{array}
\usepackage{multirow}
\usepackage{caption}

\newtheorem{thm}{Theorem}
\newtheorem{cor}[thm]{Corollary}
\newtheorem{lem}[thm]{Lemma}
\newtheorem{prop}[thm]{Proposition}

\theoremstyle{definition}

\theoremstyle{remark}

\newcommand{\eps}{\varepsilon}

\newcommand{\scrQ}{\mathscr{Q}}
\newcommand{\scrP}{\mathscr{P}}
\newcommand{\scrO}{\mathscr{O}}

\newcommand{\scrV}{\mathscr{V}}
\newcommand{\scrR}{\mathscr{R}}
\newcommand{\scrT}{\mathscr{T}}

\newcommand{\scrl}{\ell}
\newcommand{\bfv}[1][]{{v^{#1}}}
\DeclareMathOperator{\vol}{vol} 

\def\mystrut(#1,#2){\vrule height #1pt depth #2pt width 0pt}

\journal{Discrete Applied Mathematics}

\begin{document}

\begin{frontmatter}

\title{Volume computation for sparse\\ boolean quadric relaxations}%

\author[1]{Jon Lee}%
\address[1]{Dept. of I\&OE, University of Michigan, Ann Arbor, MI, USA.}
\author[2]{Daphne Skipper}%
\address[2]{Dept. of Mathematics, U.S. Naval Academy, Annapolis, Maryland, USA.}

%


%

%

\begin{abstract}
Motivated by understanding the quality of tractable convex
relaxations of intractable polytopes, Ko et al. gave a closed-form expression for the volume of a standard relaxation $\scrQ(G)$ of the boolean quadric polytope (also known as the (full) correlation polytope) $\scrP(G)$ of the complete graph $G=K_n$.
We extend this work to structured sparse graphs. In particular, we (i) demonstrate the existence of an efficient algorithm for $\vol(\scrQ(G))$ when $G$ has bounded treewidth, (ii) give closed-form expressions (and asymptotic behaviors)
for $\vol(\scrQ(G))$ for all stars, paths, and cycles,
and (iii) give a closed-form expression for $\vol(\scrP(G))$ for all cycles.
Further, we demonstrate that when $G$ is a cycle, the simple relaxation $\scrQ(G)$ is a very close model for the much more complicated $\scrP(G)$.
Additionally, we give some computational results demonstrating
that this behavior of the cycle seems to extend to more complicated graphs. Finally, we speculate on the possibility of extending
some of our results to cactii or even series-parallel graphs.
\end{abstract}

\begin{keyword}
volume \sep boolean quadric polytope \sep correlation polytope \sep mixed-integer non-linear optimization \sep order polytope \sep counting linear extensions \sep
bounded treewidth \sep cut polytope



\MSC[2010] 52B11 \sep 52A38 \sep 90C26 \sep 90C10

\end{keyword}

\end{frontmatter}


\section{Introduction}\label{sec:intro}
For a simple undirected graph $G=(V,E)$ with vertex set $V:=\{1,2,\ldots,n\}$ and edge set $E\subset \{(i,j) ~:~ 1\leq i<j \leq n\}$, we let $m:=|E|$.
The \emph{(graphical) boolean quadric polytope} $\scrP(G)$ is the convex hull in dimension $d:= n+\binom{n}{2}$ of the set of binary solutions $\{x_i, y_{ij}: i \in V, i<j \in V\}$ to the system:
\[
x_ix_j=y_{ij} \mbox{, for each edge }  (i,j) \in E.
\]
 When $G$ is the complete graph $K_n$,
we have the well-known boolean quadric polytope (or \emph{(full) correlation polytope}) $\scrP(K_n)$.

  The polytope $\scrP(G)$ is contained in and naturally modelled by $\scrQ(G)$, the solution set in $\mathbb{R}^d$ of the linear inequalities
  \begin{align}
  y_{ij}  &~\geq~  0, \tag{F0}\label{facet0} \\
y_{ij}  &~\leq~  x_i, \tag{F1}\label{facet1} \\
y_{ij}  &~\leq~  x_j,  \tag{F2}\label{facet2} \\
x_i + x_j  &~\leq~  1 + y_{ij},  \tag{F3}\label{facet3}
\end{align}
for $(i,j) \in E$, in which the variables are now relaxed to real values,
which here amounts to the continuous interval $[0,1]$.

Padberg heavily investigated these fundamental polytopes, describing families of facet-describing inequalities, the affine equivalence of these polytopes with the cut polytope
(and its relaxations)
of a complete graph, and much more (see \cite{Padberg1989}); also see \cite{BQ4} and \cite{BQ5}.
Under the name of correlation polytopes, further early work concerning geometry and complexity appeared as \cite{Pitowsky1991};
also see \cite{DezaLaurent}.
The boolean quadric polytope and related polytopes are fundamental to state-of-the-art approaches (both exact branch-and-bound methods and approximation algorithms) for NP-hard max-cut problems; see \cite{Wiegele2010} and \cite{GW1995}. Improvements in how the
geometry can be exploited computationally have recently been discovered; see \cite{BQ1}.
 Additionally, the boolean quadric polytope and its relaxations
are fundamental to a recent successful approach to general box-constrained quadratic-optimization problems; see  \cite{BGL}. Relatives of the boolean quadric
polytope appear in other combinatorial problems; see  \cite{BQ3} and   \cite{BQ2}, for example.
 Recent work on the currently  hot area of extension complexity (for example, see \cite{CCZ2010}) showed that these polytopes do not have compact extended formulations (i.e., they are not projections of polytopes with a polynomial number of facets and variables); see \cite{Fiorini2015} and \cite{Kaibel2015}. An implication of Padberg's work is
that there is an efficient algorithm to decide if a linear function is optimized over $\scrQ(G)$ by some integer extreme point.
In very recent work, \cite{Nikolaev2016} studied a natural generalization of $\scrQ(G)$ and found that the corresponding problem is NP-complete.

 Although we concentrate on the boolean quadric polytope $\scrP(G)$ of a graph $G$
(and its relaxations), it is well known that $\scrP(G)$ is intimately related, via a linear transformation, to the cut polytope of the graph $G+u$:
 for which a new vertex $u$ is joined to every vertex of $G$ (see \cite{BQ4}).
 In fact, the determinant of that linear transformation is known (see \cite{KLS1997}), and so every volume result concerning $\scrP(G)$
 and its relaxations immediately gives a volume result
 concerning the cut polytope of $G+u$ and related relaxations.

Lee and Morris (see \cite{LM1994}) introduced the idea of using volume as a measure for evaluating relaxations of combinatorial polytopes. They were particularly
motivated by the now hot area (for example, see \cite{LeeLeyffer}) of optimizing nonlinear functions
over polytopes. Lee and Morris obtained results comparing uncapacitated
facility-location polytopes with their natural relaxations and
stable-set polytopes with their relaxations (also see  \cite{Steingr}). In particular, \cite{LM1994} demonstrated that
in a precise asymptotic sense, natural relaxations should be
adequate when the number of customers increases much faster than the number
of potential facilities.
Their work on volume-based comparison for facility-location relaxations was borne out in computational experiments
done with later-available convex mixed-integer nonlinear-optimization solvers
like \verb;Bonmin; (see \cite{Lee2007}). In other recent work of this type,
volume was used to engineer an aspect of the ``spatial-branch-and-bound''
approach to non-convex mixed-integer nonlinear-optimization; see
 \cite{SpeakmanLee2015},  \cite{SpeakmanYuLee}, \cite{SpeakmanLee_Branching}.

In \cite{KLS1997}, Ko, Lee, and Steingr{\'{\i}}msson
 established that the $d$-dimensional volume of $\scrQ(K_n)$ is $2^{2n-d}n!/(2n)!$.
 This work can be seen as a key first step in comparing $\scrQ(K_n)$ with $\scrP(K_n)$ via volume.
 Unfortunately, we still do not have a good estimate of the volume of $\scrP(K_n)$, but the suspicion is that it is
 substantially smaller than that of $\scrQ(K_n)$.

At the other extreme, as compared to $K_n$, when $G$ has no edges, $\scrQ(G)=[0,1]^d$, the $d$-dimensional unit hypercube $\{(x,y) ~:~ x\in [0,1]^n, y \in [0,1]^{\binom{n}{2}}\}$.
In what follows, we consider the volume of $\scrQ(G)$ for cases when $0<|E(G)| < \binom{n}{2}$. In particular, we focus our attention on sparse $G$.

Our interest in studying the boolean quadric polytope and its relaxations for sparse graphs, from the viewpoint of volumes,
is based on: (i) the fundamental role that the boolean quadric polytope
has been playing in combinatorial optimization through recent times,
(ii) the generally easier tractability of combinatorial-optimization problems on sparse graphs (see \cite{BK2008}),
and (iii) the recent success of using  volume as a means of comparing relaxations.

In our volume calculations for $\scrQ(G)$, we are able to consider inequalities \ref{facet1} and \ref{facet2} (largely) independently from \ref{facet3}.  For convenience, we define an additional (well-known)
polytope arising from $G$.  Let $\scrO(G)$ denote the \emph{order polytope} of $G$ (see \cite{Stanley1986}): the subset of $[0,1]^d$ satisfying inequalities of the form \ref{facet1} and \ref{facet2} (but not necessarily \ref{facet3}) for edges $(i,j) \in E$.

We use the notation $\vol_d()$ to denote the $d$-dimensional volume of a convex body in $\mathbb{R}^d$.
We first observe a very simple and useful formula for the volumes of the polytopes that we associate with a graph $G$, given the volumes of the polytopes associated with the connected components of $G$.

\begin{lem} \label{crossproduct}
Let $G = (V,E)$ be a simple graph with connected components $G_i = (V_i, E_i)$, $i = 1, 2, \dots, k$.  Then
\[
\vol_d(\mathscr{X}(G)) = \prod_{i=1}^k \vol_{d_i}(\mathscr{X}(G_i)),
\]
where $\mathscr{X} \in \left\{\scrP, \scrQ, \scrO\right\}$, $d = |V| + \binom{|V|}{2}$, and $d_i = |V_i| +  |E_i|$.
\end{lem}

\begin{proof}
If $i < j \in V$, but $(i,j) \notin E$, the variable $y_{i j}$ appears in no inequalities defining $\mathscr{X}(G)$ other than $0 \leq y_{ij} \leq 1$.  Each of these variables contribute a unit multiplier to the $d$-dimensional volume calculation of $\mathscr{X}(G)$; in other words, we can ignore them in our volume formulae.

Because the connected components of $G$ share no common vertices or edges, the polytopes $\mathscr{X}(G_i), i \in \{1, 2, \dots, k\},$ are defined on pairwise disjoint sets of variables. The result then easily follows by noting that
volumes multiply under cross products.
\end{proof}

\begin{cor}
If $G$ is a matching with $m\geq 1$ edges, then $\vol(G)=1/6^m$.
\end{cor}

\begin{proof}
Follows easily from Lemma \ref{crossproduct} and Ko et al.'s formula applied to the one-edge graph $K_2$.
\end{proof}

 Note that Lemma \ref{crossproduct} allows us to mostly restrict our focus to connected graphs.
Even when we consider graphs with multiple
connected components, we can omit $x$ variables for isolated vertices and $y$-variables that do not represent edges of $G$ without affecting volumes.
 In this way, it really does not matter whether we regard
our polytopes to be in $\mathbb{R}^{n+\binom{n}{2}}$ or $\mathbb{R}^{n+m}$,
so often we will omit the dimension and just write $\vol(\cdot)$.
An exception to this is when we carry out asymptotic analysis, see \S\ref{sec:asymp}, where the dimension of the ambient space is important.

We conclude this section with a brief overview of the rest of the paper.
In \S\ref{sec:order}, we mildly extend the relationship from \cite{KLS1997}
between $\scrQ(G)$ and $\scrO(G)$ for a graph $G$, and point out how for graphs of bounded treewidth (e.g., series-parallel graphs), we can then compute $\vol(\scrQ(G))$ in polynomial time.
In the three sections that follow, we develop closed-form expressions for simple families of sparse graphs
(in particular, stars, paths and cycles). In doing so, we are then able
to carry out some asymptotics to gain useful insights. Furthermore, our analyses
 serve to demonstrate how complicated the solution can be even for
 very simple classes of sparse graphs, giving some indication that
 a non-trivial algorithm really is needed for graphs of bounded treewidth.
In \S\ref{sec:stars}, we give a closed-form expression for $\vol(\scrQ(G))$, when $G$ is a star.
In \S\ref{sec:paths}, we give a closed-form expression for $\vol(\scrQ(G))$, when $G$ is a path.
Stars and paths are of course forests, and so in those cases, as established by Padberg (\cite[Proposition 8]{Padberg1989}), we have that $\vol(\scrQ(G))=\vol(\scrP(G))$.
In \S\ref{sec:cycles}, we give closed-form expressions for $\vol(\scrQ(G))$ \emph{and} $\vol(\scrP(G))$, when $G$ is a cycle.
We note that our closed-form expressions involve factorial and Euler numbers.
For information about calculating them efficiently, we refer the
reader to \cite{Borwein1985}, \cite{BrentHarvey2011}, and \cite{Farach-Colton2015}.
In \S\ref{sec:asymp}, we make asymptotic analyses of the formulae that we have.
In particular, we find that $\vol(\scrQ(G))$ and $\vol(\scrP(G))$ are quite close when $G$ is a cycle.
In \S\ref{sec:exper}, we report on computational experiments designed to see how
the behavior of cycles may show up  for more complex graphs.
In \S\ref{sec:conclusions}, we describe an avenue for further investigation.

Throughout, we assume familiarity with basic notions in graph theory (see \cite{West}, for example)
and polyhedral theory (see \cite{LeeCombOpt}, for example).
But briefly, we will summarize important facts and terminology related to polytopes.
As is common, for a polytope $P\subset \mathbb{R}^d$, the \emph{dimension} $\dim(P)$ is the maximum number of
affinely independent points in $P$, minus 1. If $\dim(P)=d$, then $P$ is \emph{full dimensional}.
An inequality $\alpha'x \leq \beta$ is \emph{valid} for $P$ if
is  satisfied by all points in $P$.
The \emph{face} of $P$ \emph{described by} the valid inequality $\alpha'x \leq \beta$  is
$\{x\in P ~:~ \alpha'x = \beta\}$. Faces of polytopes are again polytopes, and the faces
of dimension one less than that of $P$ are its \emph{facets}. For a full-dimensional polytopes $P$,
 there are a finite number of facets,
each facet-describing inequality is unique up to a positive scaling,
and $P$  is the solution set of its facet-describing inequalities.


 \section{Order polytopes} \label{sec:order}

The following two results reduce the problem of calculating the volumes of $\scrO(G)$ and $\scrQ(G)$ to that of counting the number of linear extensions of a certain poset (partially-ordered set).  In particular, let $(\scrV(G), \prec)$ denote the poset on
\[
\scrV(G) =  \{x_i ~:~ i \in V\} \cup \{y_{ij} ~:~ (i,j) \in E\},
\]
with $y_{ij} \prec x_i$ and $y_{ij} \prec x_j$, for all edges $(i,j) \in E$. This poset is known as the \emph{incidence poset} of $G$.
 Let $e(\scrV(G), \prec)$ denote the number of linear extensions of $(\scrV(G), \prec)$; i.e., the number of order-preserving permutations of $\scrV(G)$.

\begin{thm} \label{le_to_volO} Let $G=(V,E)$ be a simple graph.  Then
\[
\vol(\scrO(G)) = \frac{e(\scrV(G), \prec)}{d!},
\]
where $d = |V|+|E|$.
\end{thm}

\begin{proof}
Our polytope $\scrO(G)$ is an order polytope as described by Stanley in \cite{Stanley1986}; this result follows directly from Corollary 4.2 in the same paper.
\end{proof}
\bigskip

\begin{thm} \label{volO_to_volQ}
Let $G=(V,E)$ be a simple graph.  Then
\[
\vol(\scrQ(G)) ~=~ \frac{\vol(\scrO(G))}{2^{|E|}} ~=~ \frac{e(\scrV(G), \prec)}{d!2^{|E|}} ,
\]
where $d = |V|+|E|$.
\end{thm}

\begin{proof}
We proceed as in the proofs of Proposition 1, Corollary 2, and Proposition 3 in \cite{KLS1997}.  The primary difference in our case is that we omit variables $y_{ij}$ for which $(i,j) \notin E$.

Define $\hat{\scrQ}(G) := 2\scrQ(G)$ to be the solution set of
\begin{align}
y_{ij}  &~\leq~  x_i,  \label{ineq1}\\
y_{ij}  &~\leq~  x_j,  \label{ineq2}\\
x_i + x_j  &~\leq~  2 + y_{ij}, \label{ineq3} \\
y_{ij} &~\geq~ 0,
\end{align}
for edges $(i,j) \in E$.  It is clear that  $\vol(\hat{\scrQ}(G))=2^d \vol(\scrQ(G))$, where $d = |V|+|E|$.  Via the same argument as in \cite{KLS1997}, we partition $\hat{\scrQ}(G)$ into $2^{|V|}$ equi-volume polytopes:
\[
R_a := \{(x,y) \in \hat{\scrQ}(G): a \leq x \leq a + \mathbf{1}\}, a \in \{0,1\}^{|V|},
\]
where $\mathbf{1}$ is the $d$-vector $(1,1,\dots,1)$.  In the case of $R_0$, the inequality \eqref{ineq3} is rendered vacuous, so that $R_0$ is an order polytope with
\[
\vol(R_0) = \frac{e(\scrV(G), \prec)}{d!}.
\]
We conclude that

\[
\vol(\scrQ(G))
~=~ \frac{\vol(\hat{\scrQ}(G))}{2^d}
~=~ \frac{2^{|V|} \vol(R_0)}{2^d}
~=~ \frac{e(\scrV(G), \prec)}{2^{|E|} d!}.
\]

\end{proof}

\bigskip

To find the volumes of the relaxation polytopes $\scrO(G)$ and $\scrQ(G)$, we must count the number of linear extensions of $(\scrV(G), \prec)$.  In general, counting the number of linear extensions of a poset is {\#}P-complete; see \cite{BrWi}.
We are particularly interested in situations when
counting the number of linear extensions of $(\scrV(G), \prec)$
is easier, due to some structured sparsity of $G$.

For \emph{any} poset $(N, \prec)$, we can consider its  \emph{directed cover graph}
 $\mathcal{DC}(N, \prec)$, having vertex set $N$ and an edge from vertex $i$ to distinct vertex $j$ when $i \prec j$ and there is no $k$, distinct from $i$ and $j$, with $i \prec k \prec j$. We let $\mathcal{C}(N, \prec)$
 denote the \emph{cover graph}, ignoring the edge
 directions in $\mathcal{DC}(N, \prec)$.

Cover graphs of incidence posets have been studied (see \cite{Trotter2014}, for example).
 We are interested in sparsity properties of  $\mathcal{C}(G):=\mathcal{C}(\scrV(G), \prec)$, the cover graph of the incidence poset of $G$.
  See Figure \ref{covergraph} for an example of a simple graph $G$ and the related directed cover graph $\mathcal{DC}(G)$.

It is clear that $\mathcal{C}(G)$ inherits all of the graph properties of $G$ that are inherited by edge subdivision.  For example, if $G$ is a tree, cycle, cactus, or series-parallel graph, then $\mathcal{C}(G)$ is a tree, cycle, cactus, or series-parallel graph, respectively.
Furthermore, if $G$ has treewidth bounded by $k$, then
 $\mathcal{C}(G)$ has treewidth bounded by $k$:
 whenever $G$ has treewidth $k$, then inserting a vertex on each edge leaves the treewidth at $k$ (this is mentioned and used in many papers, but see \cite[Lemma 1]{Lozin} where it is explicitly stated and proved).
 Besides trees ($k=1$), cycles ($k=2$), cactii ($k=2$), and series-parallel graphs ($k=2$), we can consider outer-planar graphs ($k=2$) and Halin graphs ($k=3$).

Structured sparsity of  $\mathcal{C}(G)$ can be exploited.
\cite{Atkinson} and \cite{Atkinson1990} give polynomial-time algorithms for counting the number of linear extensions when $\mathcal{C}(N, \prec)$ is a tree.  \cite{Chang}  extended this to the case in which $\mathcal{C}(N, \prec)$ is a cactus (also see\cite{Atkinson}). \cite{Kangas2016} extended this to to the case in which $\mathcal{C}(N, \prec)$ has bounded treewidth.  We have then the following fundamental result.

\begin{thm}
For the class of graphs of treewidth bounded by any constant, in  polynomial time, we can calculate $\vol(\scrO(G))$ and $\vol(\scrQ(G))$.
\end{thm}

Interestingly,  \cite{BrWi} proved that
 counting the number of linear extensions of a poset is {\#}P-complete
 even for posets of height 3, and they conjectured that this is
 true even for posets of height 2. The height-2 situation is very relevant to our investigation because our posets $(\scrV(G), \prec)$ have height 2.
 However, our posets are rather special posets of height 2, in that all of our $y$ vertices have degree 2, so a positive complexity result is more likely. In any case,
Brightwell and  Winkler asserted\footnote{January 16, 2017, private communication.} that: (i) the complexity for the general height-2 case is still open; (ii) there seems to be no work on counting linear extensions of incidence posets; (iii)
there is no compelling reason to believe that the case of incidence posets should be easier than general height-2 posets.


\begin{figure}
\begin{center}
        \begin{tikzpicture}
          [font=\scriptsize,
          node/.style={shape=circle,draw=black,fill=white!20, text=black,minimum width=0.5cm,thick},
          arc/.style={->,>=stealth,thick},
          edge/.style={thick}]

            \node (1) [node] at (-1,0) {1};
            \node (2) [node] at (1,0) {2};
            \node (3) [node] at (3,1.5) {3};
            \node (4) [node] at (3,-1.5)  {4};

         \draw [edge] (1) to node [auto] { } (2);
         \draw [edge] (2) to node [auto] { } (3);
         \draw [edge] (3) to node [auto] { } (4);
         \draw [edge] (2) to node [auto] { } (4);


            \node (5) [draw] at (5,0) {$x_1$};
            \node (6) [draw] at (8,0) {$x_2$};
            \node (7) [draw] at (11,2) {$x_3$};
            \node (8) [draw] at (11,-2)  {$x_4$};
    \node (56) [draw] at (6.5,0) {$y_{12}$};
    \node (67) [draw] at (9.5,1)  {$y_{23}$};
    \node (78) [draw] at (11,0)  {$y_{23}$};
    \node (68) [draw] at (9.5,-1)  {$y_{23}$};

         \draw [arc] (5) to node [auto] { } (56);
         \draw [arc] (6) to node [auto] { } (56);
         \draw [arc] (6) to node [auto] { } (67);
         \draw [arc] (7) to node [auto] { } (67);
         \draw [arc] (7) to node [auto] { } (78);
         \draw [arc] (8) to node [auto] { } (78);
         \draw [arc] (6) to node [auto] { } (68);
         \draw [arc] (8) to node [auto] { } (68);

        \end{tikzpicture}
\end{center}
\caption{$G$ and $\mathcal{DC}(G)$}  \label{covergraph}
\end{figure}
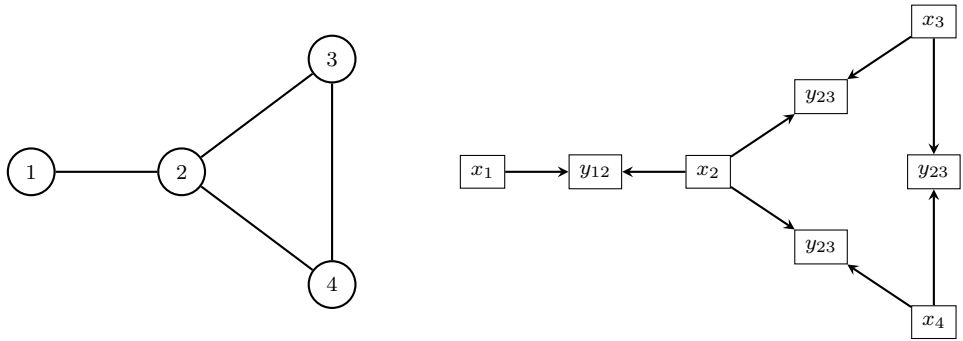


\section{Stars}\label{sec:stars}

Let $S_m$ denote a star with $m \geq 1$ edges.

\begin{lem}\label{lem:starLE}
For $m\geq 1$,
\[
e(\scrV(S_m),\prec) = 2^m(m!)^2.
\]
\end{lem}

\begin{proof}
Let the vertex set of the star be $V:=\{0,1,\ldots,m\}$, and let vertex $0$ be the center of the star.
For convenience, we count the reverse linear extensions of $e(\scrV(S_m),\prec)$, in which $x_0$ and $x_k$ appear to the left of $y_{0,k}$.

For $i = 0, 1, \dots,m$, the number of permissible permutations of $\scrV(S_m)$ in which $x_0$ appears in position $i+1$ and all other $x$ variables are ordered by index is given by
\[
\left(i! \binom{2m-i}{i} \right) \left(\frac{ (2(m-i))!}{2^{m-i}(m-i)!} \right).
\]
The first factor counts the number of ways to place $y_{0,1}, y_{0,2}, \dots, y_{0,i}$ into the $2m-i$ positions to the right of $x_0$ in no particular order.  The second factor counts the placement of the pairs $x_k, y_{0,k}$, for $i+1 \leq k \leq m,$ in the remaining $2(m-i)$ positions to the right of $x_0$ with $x_k \prec y_{0,k}$ and the $x_k$ ordered by index.

Incorporating all possible positions of $x_0$ and permutations of the $x_k, y_{0,k}$ pairs for $1 \leq k \leq m$, we obtain
\begin{align*}
e(\scrV(S_m),\prec)  &~=~ m! \sum_{i=0}^m  i! \binom{2m-i}{i} \left(\frac{ (2(m-i))!}{2^{m-i}(m-i)!} \right) \\
&~=~ m! \sum_{i=0}^m \frac{(2m-i)!}{(m-i)!2^{m-i}} \\
&~=~ 2^m (m!)^2 \sum_{i=0}^m \binom{2m-i}{m-i} \frac{1}{2^{2m-i}} \\
&~=~ 2^m (m!)^2 \sum_{\scrl=0}^m \binom{m+\scrl}{\scrl} \frac{1}{2^{m+\scrl}}.
\end{align*}
To see that $\sum_{\scrl=0}^m \binom{m+\scrl}{\scrl} \frac{1}{2^{m+\scrl}} = 1$, rewrite the summation as
\[
\sum_{\scrl=0}^m \binom{m+\scrl}{\scrl}\left[ \left(\frac{1}{2}\right)^{m+1}\left(\frac{1}{2}\right)^{\scrl} + \left(\frac{1}{2}\right)^{m+1}\left(\frac{1}{2}\right)^{\scrl}\right],
\]
which represents the probability that a fair coin will land on the same side (heads or tails) exactly $m+1$ times somewhere between flips $m+1$ and $2m+1$.  This identity is a special case (in which $x= \frac{1}{2}$) of an identity apparently due to Gosper  (see \cite[Item 42]{HAKMEM1972}).
\end{proof}

Combining Lemma \ref{lem:starLE} and Theorem \ref{volO_to_volQ}, we obtain the following result.
\begin{thm}\label{thm:starvol}
For $m \geq 1$,
\[
\vol(\scrQ(S_m)) ~=~
\frac{\vol(\scrO(S_m))}{2^m} ~=~
\frac{(m!)^2}{(2m+1)!}.
\]
\end{thm}


\section{Paths}\label{sec:paths}

We will see that the so-called odd ``Euler numbers''
appear in the formulae for the volumes of polytopes associated with paths (and cycles).
There are several
closely related sequences that are called ``Euler numbers'', so there can be considerable confusion.
Following \cite{Stanley2012},
an \emph{alternating permutation} of $\{1, 2,\ldots, k\}$ is a permutation so that each entry is alternately greater or less than the preceding entry. The \emph{Euler number} $A_k$, $k\geq 1$, is the number of such alternating permutations,
and \emph{Andr\'e's problem} is   determining the $A_k$.
For small values, we have
\[
A_0:=1,~A_1=1,~A_2=1,~A_3=2,~A_4=5,~A_5=16,~A_6=61,~A_7=272.
\]

Even-indexed Euler numbers are also called \emph{zig numbers}, and the odd-indexed ones are called \emph{zag numbers}.
Also, the even-indexed ones are called \emph{secant numbers}, and the odd-indexed ones are called \emph{tangent numbers}.
The latter names come from Andr\'e who established the following pretty result.
\begin{thm}[D. Andr\'e, \cite{andre}]\label{thm:andre}
The Maclaurin series for $\sec(x)$ is $\sum_{m=0}^\infty \frac{A_{2m}}{(2m)!}x^{2m}$, and the
Maclaurin series for $\tan(x)$ is $\sum_{m=0}^\infty \frac{A_{2m+1}}{(2m+1)!}x^{2m+1}$, both having
radius of convergence $\pi/2$.
\end{thm}

Note that Euler worked with the odd-indexed ones, and defined them not combinatorially but rather
via the Maclaurin series for $\tan(x)$ (see \cite{Stanley2012}).

Let $P_m$ denote a path with $m\geq 0$ edges.

\begin{lem} \label{le_path}
For $m\geq 0$,
\[
e(\scrV(P_m),\prec) = A_{2m+1}.
\]
\end{lem}

\begin{proof}
Figure \ref{pathcover} is a diagram of $\mathcal{C}(P_m)$.  The number $e(\scrV(P_m),\prec)$ counts the number of permutations $\phi$ of
\[
\{x_1,~ y_{1,2}, ~x_2, ~y_{2,3}, ~x_3, ~\dots, ~x_{m}, ~y_{m, m+1}, ~x_{m+1}\},
\]
such that
\[
\phi(x_1) > \phi(y_{1,2}) < \phi(x_2) > \phi(y_{2,3}) < \dots  > \phi(y_{m,m+1}) < \phi(x_{m+1}).
\]

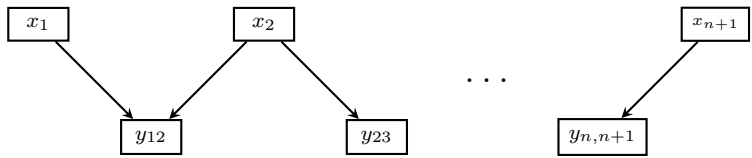
\begin{figure}
\begin{center}
        \begin{tikzpicture}
          [font=\scriptsize,
          node/.style={draw=black,fill=white!20, text=black,minimum width=0.8cm,thick},
          arc/.style={->,>=stealth,thick},
          edge/.style={thick}]
	   /.style={text=black, minimum width=0.5cm, thick}

            \node (x1) [node] at (0,0) {$x_1$};
            \node (x2) [node] at (3,0) {$x_2$};
          \node (xn1) [node] at (9,0)  {\tiny $x_{n+1}$};
   	    \node (y12) [node] at (1.5,-1.5) {$y_{12}$};
   	    \node (y23) [node] at (4.5,-1.5) {$y_{23}$};

   	    \node (ynn1) [node] at (7.5,-1.5) {$y_{n,n+1}$};

         \draw [arc] (x1) to node [auto] { } (y12);
         \draw [arc] (x2) to node [auto] { } (y12);
         \draw [arc] (x2) to node [auto] { } (y23);
         \draw [arc] (xn1) to node [auto] { } (ynn1);
	 \node  [draw=white] at (6, -.75) {\large $\dots$};
        \end{tikzpicture}
\end{center}
\caption{$\mathcal{C}(P_m)$}  \label{pathcover}
\end{figure}
\noindent In other words, $e(\scrV(P_m), \prec)$ is precisely the number of  alternating permutations  of a $(2m+1)$-element set, which is given by the odd Euler number, $A_{2m+1}$.
\end{proof}

Combining Lemma \ref{le_path} and Theorems \ref{le_to_volO} and \ref{volO_to_volQ}, we obtain the following   result.

\begin{thm} \label{thm:pathvol}
For $m\geq 0$,
\[
\vol(\scrQ(P_m)) ~=~  \frac{\vol(\scrO(P_m))}{2^m}  ~=~ \frac{A_{2m+1}}{2^m(2m+1)!}.
\]
\end{thm}

Because of Andr\'e's Theorem, we can see the volumes of the boolean quadric polytopes corresponding to paths in the McLaurin series expansion of tangent.

\begin{cor}
\[
\sum_{m \geq 0} 2^m \vol(\scrQ(P_m))x^{2m+1} ~=~ \tan(x).
\]
\end{cor}

\section{Cycles}\label{sec:cycles}

In this section we obtain a first success at fully analyzing a
situation where $\scrQ(G)$ is different from $\scrP(G)$.  Let $C_m$ be a simple cycle with $m\geq 3$ edges.

\begin{thm}[Kreweras, \cite{Kreweras}]\label{lem:cycleLE}
For $m \geq 3$,
\[
e(\scrV(C_m),\prec) = mA_{2m-1}.
\]
\end{thm}

\begin{proof}
The number of cyclically alternating permutations of length $2m$ is precisely $mA_{2m-1}$; see
\cite{Kreweras}.
\end{proof}

Combining Theorem \ref{lem:cycleLE} and Theorem \ref{volO_to_volQ}, we obtain the following result.

\begin{thm} \label{vol_QCm}
For $m \geq 3$,
\[
\vol(\scrQ(C_m)) ~=~
\frac{\vol(\scrO(C_m))}{2^m} ~=~
\frac{mA_{2m-1}}{2^m(2m)!}.
\]
\end{thm}

Combining Theorems \ref{vol_QCm} and \ref{thm:pathvol}, we obtain the following result.
\begin{cor}
For $m \geq 3$,
\[
\vol(\scrQ(C_m)) ~=~
\vol(\scrQ(P_{m-1}))\big/ 4.
\]
\end{cor}

 Padberg made a careful analysis of the facets of $\scrP(G)$
when $G$ is a cycle (see \cite{Padberg1989}). We summarize the relevant parts in the remainder of this paragraph.
Let $G=(V,E)$ be a simple graph containing simple cycle $C = (V(C), E(C))$.
 Let $A$ be an odd cardinality subset of $E(C)$ and define
\begin{align*}
	S_0 &~:=~ \{i \in V(C): i \text{ is incident to no elements of } A\}; \\
	S_1 &~:=~ \{i \in V: i \notin V(C) \text{ or } i \text{ is incident to exactly $1$ element of } A\}; \\
	S_2 &~:=~ \{i \in V(C): i \text{ is incident to $2$ elements of } A\}.
\end{align*}
Note that $V = S_0 \cup S_1 \cup S_2$.
\begin{thm}[Padberg \cite{Padberg1989}]\label{thm:oddpad}
The \emph{odd cycle inequality} $OC(A)$,
\begin{align}
	\sum_{i \in S_2}x_i ~-~ \sum_{i \in S_0}x_i ~+ \sum_{(i,j) \in E\backslash A}y_{ij} ~- \sum_{(i,j) \in A} y_{ij} 		&~\leq~ \Bigl\lfloor |A|/2 \Bigr\rfloor,  \tag{F4}\label{oddcyclefacet}
\end{align}
is a valid inequality of $\scrP(G)$ and cuts off the vertex $\bfv[0]$ of $\scrQ(G)$ given by
\begin{align*}
	x_i &= \textstyle \frac{1}{2}, \text{ for } i \in V, \\
	y_{ij} &= \textstyle \frac{1}{2}, \text{ for } (i,j) \in E\backslash A, \text{ and}\\
	y_{ij} &= 0, \text{ for } (i,j) \in A \cup (E\backslash E(C)).
\end{align*}
Moreover, when $G$ is a cycle, $\scrP(G)$ is the solution set of \ref{facet0}-\ref{oddcyclefacet}, all describing unique facets.
\end{thm}

In fact, $\scrP(G)$ is
completely described by \ref{facet0}-\ref{oddcyclefacet}
for every series-parallel graph $G$ (see \cite[Theorem 10]{Padberg1989}).

Next, we look carefully at the parts of $\scrQ(G)$
cut off by distinct odd cycle inequalities arising from a single simple cycle $C$ of $G$.
We will see that they are disjoint, and for the special case that $G = C$, the cut off parts all have the same
volume which we can calculate.  Note that this is a similar behavior
to so-called ``clipping inequalities'' applied to the standard unit hypercube (see \cite{CopperLee}).

\begin{lem} \label{lem:disj}
Let $G$ be a simple graph containing simple cycle $C$.  Let $A$ and $B$ be distinct odd-cardinality subsets of $E(C)$.  The odd cycle inequalities $OC(A)$ and $OC(B)$ remove disjoint portions of $\scrQ(G)$.
\end{lem}

\begin{proof}
Let $A$ and $B$ be distinct odd-cardinality subsets of $E(C)$.
Let $S_0, ~S_1$, and $S_2$ be as defined above for $A$.  Let $T_0,~ T_1,$ and $T_2$ be the corresponding subsets of $V$ related to $B$.
Suppose that $(\hat{x},\hat{y})\in \scrQ(G)$ violates both $OC(A)$ and $OC(B)$. That is,
\[
	\sum_{i \in S_2}\hat{x}_i ~-~ \sum_{i \in S_0}\hat{x}_i ~+ \sum_{(i,j) \in E\backslash A}\hat{y}_{ij} ~- \sum_{(i,j) \in A}\hat{y}_{ij} 		~>~ \Bigl\lfloor |A|/2 \Bigr\rfloor
\]
and
\[
	\sum_{i \in T_2}\hat{x}_i ~-~ \sum_{i \in T_0}\hat{x}_i ~+ \sum_{(i,j) \in E\backslash B}\hat{y}_{ij} ~- \sum_{(i,j) \in B} \hat{y}_{ij} 		~>~ \Bigl\lfloor |B|/2 \Bigr\rfloor.
\]
Adding these inequalities, canceling terms with opposite signs, we have
\begin{multline} \label{sumofOCs}
\sum_{i \in S_2 \cap T_1}\hat{x}_i ~+~ \sum_{i \in S_1 \cap T_2}\hat{x}_i ~+~ 2\sum_{i \in S_2 \cap T_2}\hat{x}_i ~-~ 2\sum_{(i,j) \in A \cap B}\hat{y}_{ij} \\
~-~ \sum_{i \in S_0 \cap T_1}\hat{x}_i  ~-~ \sum_{i \in S_1 \cap T_0}\hat{x}_i ~-~ 2\sum_{i \in S_0 \cap T_0}\hat{x}_i ~+~ 2\sum_{(i,j) \in E\backslash (A \cup B)}\hat{y}_{ij} \\
 ~>~ \frac{|A|+|B|}{2} -1.
\end{multline}

Because $(\hat{x},\hat{y})\in \scrQ(G)$, we can employ $\ref{facet0}-\ref{facet3}$. Starting with \eqref{sumofOCs},
add $\hat{y}_{ij} - \hat{x}_i -\hat{x}_j \geq -1$ for $(i,j) \in A \cap B$, and 
$\hat{x}_i - \hat{y}_{ij} \geq 0$ and
$\hat{x}_j - \hat{y}_{ij} \geq 0$
for $(i,j) \in E(C)\backslash(A\cup B)$.  All of the $\hat{x}$ terms cancel out and we obtain
\begin{align} \label{simplifiedOCs}
 \sum_{(i,j) \in A \cap B}\hat{y}_{ij}   ~>~ \frac{|A|+|B|}{2} - |A \cap B|.
\end{align}
We can apply non-negativity of the $\hat{y}_{ij}$ to eliminate the remaining $\hat{y}_{ij}$ terms, arriving at
%
\begin{align}
0 ~>~  \frac{|A|+|B|}{2} - |A \cap B| ~\geq~ \min\left\{|A|,|B|\right\} - |A \cap B| ~\geq~ 0,
\end{align}
%
a contradiction. Therefore, we cannot have
$(\hat{x},\hat{y})\in \scrQ(G)$ violating both $OC(A)$ and $OC(B)$.
\end{proof}

The following characterization of the  subset of $\scrQ(C_m)$ that is removed by an odd-cycle inequality is apparent from the proof of Theorem 9 in \cite{Padberg1989}.
\begin{lem} \label{cutoffsimplex}
(Padberg \cite{Padberg1989})  Let $A$ be any odd-cardinality subset of the edges of $C_m$.  Let $v^0$ denote the fractional vertex of $\scrQ(C_m)$ removed by $OC(A)$ (as described in Theorem \ref{thm:oddpad}).
Let the polytope $W$ be the closure of the set removed from $\scrQ(C_m)$ by $OC(A)$.

 The vertex $v^0$ is non-degenerate, and $W$  has the following properties:
\begin{itemize}
	\item the $2m+1$ vertices of $W$ are $v^0$ and the $2m$ integer vertices of $\scrP(C_m)$ that satisfy $OC(A)$ with equality;
	\item the $2m+1$ facets of $W$ are the ones described by $OC(A)$ and the set of  inequalities $\{y_{ij} \geq 0, ~x_i + x_j \leq y_{ij} + 1: (i,j) \in A\} \cup \{y_{ij} \leq x_i, ~y_{ij} \leq x_j: (i,j) \in E(C_m)\backslash A\}$;
	\item (hence) $W$ is a $2m$-dimensional simplex (in $\mathbb{R}^{2m}$).
\end{itemize}
\end{lem}

\begin{lem} \label{cutoffhalf}
The volume removed from $\scrQ(C_m)$ by a single odd-cycle constraint is $\frac{1}{2(2m)!}$.
\end{lem}

\begin{proof}
For the $m$-cycle $C_m$, we define vertex set $V := \{1,2,\dots,m\}$ and edge set $E := \{(1,2), (2,3), \dots, (m-1,m),(m,1)\}$. In the context of this proof, we adopt the order of the coordinates of $\mathbb{R}^{2m}$
 to  match the order that the corresponding vertices and edges appear in $C_m$, starting at vertex 1 and ending at edge $(m,1)$.   For clarity, we also adopt the heavier notation $y_{(i,j)} ~(:= y_{ij})$ throughout this proof.  Specifically, we consider $v \in \mathbb{R}^{2m}$ as
\[
v := (x_1, ~y_{(1,2)},~ x_2,~ y_{(2,3)}, \dots, x_{m-1},~ y_{(m-1,m)},~ x_m,~ y_{(m,1)}).
\]

Suppose that $A$ is any odd-cardinality subset of
$E$. By Lemma \ref{cutoffsimplex}, the closure of the portion of $\scrQ(C_m)$ cut off by $OC(A)$ is a simplex, which we denote by $W$.

For each edge $(i,j)$ of $C_m$, there are four facets of $W$ (those described by inequalities of type \ref{facet0}-\ref{facet3}) that define the coordinates (of points in $\scrQ(C_m)$) corresponding to edge $(i,j)$ and its end nodes, $i$ and $j$: $y_{(i,j)}$, $x_i$, and $x_j$.  We introduce the notation $\mathcal{F}0_{(i,j)}$ to indicate the facet described by an inequality of type \ref{facet0} that corresponds  to edge $(i,j)$.  We define $\mathcal{F}1_{(i,j)}$, $\mathcal{F}2_{(i,j)}$, and $\mathcal{F}3_{(i,j)}$  similarly.

Lemma \ref{cutoffsimplex} also provides the complete facet description of $W$.
In particular, there are two facets of $W$ corresponding to each $(i,j) \in E$ chosen among $\mathcal{F}0_{(i,j)}$-$\mathcal{F}3_{(i,j)}$.  The choice depends on whether or not $(i,j)$ is in $A$.  For each edge $(i,j) \in A$, $W$ has the two facets $\mathcal{F}0_{(i,j)}$ (described by $y_{(i,j)} \geq 0$) and $\mathcal{F}3_{(i,j)}$ (described by $x_i + x_j \leq 1 + y_{(i,j)}$).  For each edge $(i,j) \in \bar{A} := E \backslash A$, $W$ has the two facets $\mathcal{F}1_{(i,j)}$ (described by $y_{(i,j)} \leq x_i$) and $\mathcal{F}2_{(i,j)}$ (described by $y_{(i,j)} \leq x_j$).  These $2m$ facets, along with the facet described by $OC(A)$, are all of the facets of $W$.

As with any full-dimensional simplex, every extreme point of $W$ is defined by the intersection of all but one of its facets.  In this way, there is a one-to-one correspondence between the extreme points of $W$ and the facets of $W$.  We say that extreme point $v$ of $W$ \emph{arises by relaxing the facet $W_v$}, where $W_v$ is the single facet of $W$ that does not contain $v$.

Let $v^0$ be the extreme point of $W$ that arises by relaxing the facet described by $OC(A)$; i.e., $W_{v^0}$ is the facet described by $OC(A)$.  As noted above for facets, there are two extreme points of $W$ relating to each $(i,j) \in E$, which we denote $v^{(i,j)_{a}}$ and $v^{(i,j)_{b}}$.  It will be important later to know exactly which facet is relaxed to obtain each of these extreme points, and so we define the following:

\noindent{if $(i,j) \in A$,}
\[
W_{v^{(i,j)_{a}}} =  \mathcal{F}0_{(i,j)} ~~~\text{ and }~~~ W_{v^{(i,j)_{b}}} = \mathcal{F}3_{(i,j)};
\]
if $(i,j) \in \bar{A}$,
\[
W_{v^{(i,j)_{a}}} = \mathcal{F}1_{(i,j)}  ~~~\text{ and }~~~ W_{v^{(i,j)_{b}}} = \mathcal{F}2_{(i,j)}.
\]

We can express the volume of the simplex $W$ via the well-known formula
\[
\vol(W) = \frac{1}{(2m)!} | \det (M) |,
\]
where
\[
M:=
\left[
\begin{array}{c}
(v^{(1,2)_{a}}-v^0)' \\
(v^{(1,2)_{b}}-v^0)' \\
(v^{(2,3)_{a}}-v^0)' \\
\vdots \\
(v^{(m,1)_{b}}-v^0)'
\end{array}
\right],
\]
where $v'$ represents the transpose of column vector $v$.
Our task is to show that $\det(M) = \pm 1/2$.  In particular, we will define an upper-triangular matrix $\tilde{M}$ obtained from $M$ via certain elementary row operations that leave the determinant unchanged (replacing a row with the sum of itself and a scalar multiple of another row), and show that $\det(M) = \det(\tilde{M})= \pm 1/2$.

We turn our focus to the extreme points of $W$, which determine the rows of $M$.  For precision, we introduce superscripts to the coordinates of the extreme points of $W$ matching the
superscripts of the extreme points; e.g., $v^{(1,2)_a} := $
\[
(x^{(1,2)_a}_1, ~y^{(1,2)_a}_{(1,2)},~ x^{(1,2)_a}_2,~ y^{(1,2)_a}_{(2,3)}, \dots, x^{(1,2)_a}_{m-1},~ y^{(1,2)_a}_{(m-1,m)},~ x^{(1,2)_a}_m,~ y^{(1,2)_a}_{(m,1)})
\]
denotes the extreme point of $W$ that arises by relaxing either $\mathcal{F}0_{(1,2)}$ (if $(1,2) \in A$) or $\mathcal{F}1_{(1,2)}$ (if $(1,2) \in \bar{A}$).

By Theorem \ref{thm:oddpad}, the coordinates of $v^0$ are:
\begin{align*}
x^0_i &= \textstyle \frac{1}{2}, \text{ for } i \in V; \\[3pt]
y^0_{(i,j)} &= \left\{
	\begin{array}{ll}
		0, & \text{ for } (i,j) \in A; \\
		\frac{1}{2}, & \text{ for } (i,j) \in \bar{A}.
	\end{array} \right.
\end{align*}

By Lemma \ref{cutoffsimplex}, each of the extreme points of $W$ of the form $\bfv[(i,j)_a]$ or $\bfv[(i,j)_b]$, $(i,j) \in E$, is also an extreme point of $\scrP(C_m)$ and therefore 0/1-valued.  Focusing on the coordinates of $\bfv[(i,j)_a]$, we note that $x^{(i,j)_a}_i$, $y^{(i,j)_a}_{(i,j)}$, and $x^{(i,j)_a}_j$ must satisfy the inequality describing  $W_{\bfv[(i,j)_a]}$ with strict inequality and must satisfy the inequality describing $W_{\bfv[(i,j)_b]}$ as an equation.  A similar statement can be made about $\bfv[(i,j)_b]$, which must not be contained in $W_{\bfv[(i,j)_b]}$  and must be contained in $W_{\bfv[(i,j)_a]}$.  For every $(i,j) \in E$, the 0/1 values of the coordinates $x^{(i,j)_*}_i$, $y^{(i,j)_*}_{(i,j)}$, and $x^{(i,j)_*}_j$ of $\bfv[(i,j)_*]$ for $* \in \{a,b\}$, depend on whether or not $(i,j) \in A$ and are uniquely determined by the two facets $W_{\bfv[(i,j)_a]}$ and $W_{\bfv[(i,j)_b]}$ of $W$. Table \ref{tab:ij_values} gives these values.

\bigskip
\begin{minipage}{\linewidth}
\centering
\captionof{table}{$x^{(i,j)_*}_i$, $y^{(i,j)_*}_{(i,j)}$, $x^{(i,j)_*}_j$, $* \in \{a,b\}$} \label{tab:ij_values}
\renewcommand{\arraystretch}{1.8}
\begin{tabular}{c|c|c|c|c|c|c}
$(i,j)\in$  & $*$ & $\bfv[(i,j)_*]$ \emph{must satisfy}
& $x^{(i,j)_*}_i$ & $y^{(i,j)_*}_{(i,j)}$ & $x^{(i,j)_*}_j$ \\
\hline
\multirow{2}{*}{$A$} & $a$ &  $y^{(i,j)_a}_{(i,j)} > 0$,~ $x^{(i,j)_a}_i + x^{(i,j)_a}_j = 1+ y^{(i,j)_a}_{(i,j)} $ &  1 & 1 & 1 \\
& $b$ &  $y^{(i,j)_b}_{(i,j)} = 0$,~ $x^{(i,j)_b}_i + x^{(i,j)_b}_j < 1+ y^{(i,j)_b}_{(i,j)}$ & 0 & 0 & 0  \\
\hline
\multirow{2}{*}{$\bar{A}$} & $a$ & $y^{(i,j)_a}_{(i,j)} < x^{(i,j)_a}_i $,~ $y^{(i,j)_a}_{(i,j)} = x^{(i,j)_a}_j$ & 1 & 0 & 0  \\
& $b$ & $y^{(i,j)_b}_{(i,j)} = x^{(i,j)_b}_i $,~ $y^{(i,j)_b}_{(i,j)} < x^{(i,j)_b}_j $ & 0 & 0 & 1
\end{tabular}
\end{minipage}
\bigskip

For $(k,\ell) \in E\backslash \{(i,j)\}$ and $* \in \{a,b\}$, there are \emph{two} possible options for the coordinates $x^{(i,j)_*}_{k}$, $y^{(i,j)_*}_{(k,\ell)}$,
 and $x^{(i,j)_*}_{\ell}$ of $v^{(i,j)_*}$, allowed by the facets $W_{v^{(k,\ell)_a}}$ and $W_{v^{(k,\ell)_b}}$, which $v^{(i,j)_*}$ must be contained in.  If $(k,\ell) \in A$, then $v^{(i,j)_*}$ must satisfy \ref{facet0}$_{(k,\ell)}$ and \ref{facet3}$_{(k,\ell)}$ with equality, and if $(k,\ell) \in \bar{A}$, then $v^{(i,j)_*}$ must satisfy \ref{facet1}$_{(k,\ell)}$ and \ref{facet2}$_{(k,\ell)}$ with equality.  We summarize the possible values of  $x^{(i,j)_*}_{k}$, $y^{(i,j)_*}_{(k,\ell)}$, and $x^{(i,j)_*}_{\ell}$ in Table \ref{tab:kl_values}.

\bigskip
\begin{minipage}{\linewidth}
\centering
\captionof{table}{$x^{(i,j)_*}_k$, $y^{(i,j)_*}_{(k,\ell)}$, $x^{(i,j)_*}_{\ell}$, $(k,\ell) \neq (i,j)$, $* \in \{a,b\}$} \label{tab:kl_values}
\renewcommand{\arraystretch}{1.5}
\begin{tabular}{c|c|c|c|c}
\multirow{2}{*}{} &  & \multicolumn{3}{c}{\emph{possible values of}} \\
$(k,\ell)\in$ &
$v^{(i,j)_*}$  \emph{must satisfy}
 & $x^{(i,j)_*}_k$ & $y^{(i,j)_*}_{(k,\ell)}$ & $x^{(i,j)_*}_{\ell}$ \\
\hline
\multirow{2}{*}{$A$} &\multirow{2}{*}{$y^{(i,j)_*}_{(k,\ell)} = 0$,~ $x^{(i,j)_*}_k + x^{(i,j)_*}_{\ell} = 1+ y^{(i,j)_*}_{(k,\ell)} $} &  1 & 0 & 0 \\
 && 0 & 0 & 1 \\
\hline
\multirow{2}{*}{$\bar{A}$} & \multirow{2}{*}{$x^{(i,j)_*}_k  = y^{(i,j)_*}_{(k,\ell)} = x^{(i,j)_*}_{\ell} $}& 1 & 1 & 1 \\
  && 0 & 0 & 0 \\
\end{tabular}
\end{minipage}
\bigskip

Now we are ready to describe how we transform the matrix $M$ into an upper-triangular matrix $\tilde{M}$.  We use $A_r$ to denote row $r$ of matrix $A$.   We define
\[ \tilde{M} :=
\left[
\begin{array}{c}
M_1 \\
M_2 + M_1 \\
M_3 \pm M_2 \\
M_4 + M_3 \\
\vdots \\
M_{2m-1} \pm M_{2m-2} \\
M_{2m} + M_{2m-1}
\end{array}
\right]
=
\left[
\begin{array}{c}
(v^{(1,2)_a} - v^0)' \\
(v^{(1,2)_b} - v^0)' + (v^{(1,2)_a} - v^0)' \\
(v^{(2,3)_a} - v^0)' \pm (v^{(1,2)_b} - v^0)' \\
(v^{(2,3)_b} - v^0)' + (v^{(2,3)_a} - v^0)' \\
\vdots \\
(v^{(m,1)_a} - v^0)' \pm (v^{(m-1,m)_b} - v^0)' \\
(v^{(m,1)_b} - v^0)' + (v^{(m,1)_a} - v^0)' \\
\end{array}
\right],
\]
where the choice of $\tilde{M}_r := M_r + M_{r-1}$ versus $\tilde{M}_r:=M_r - M_{r-1}$, for $r \in \{3,5 \dots, 2m-1\}$, will be described below.  Already, we can see that $\det(\tilde{M}) = \det(M)$, and $M_{1,1} = x^{(1,2)}_1-x^{0}_1 = \pm \frac{1}{2}$.  In what follows, we completely specify the odd-indexed rows of $\tilde{M}$, and demonstrate that $\tilde{M}$ is upper-triangular with $|\tilde{M}_{r,r}| = 1$, for $r \in \{2,3,\dots,2m\}$.

$\tilde{M}$ has six different row types, which we consider in cases 1 through 6 below.  The calculation of $\tilde{M}_{r}$, for even $r = 2i  \in \{2,4,\dots,2m\}$, requires the two extreme points $v^{(i,j)_a}$ and $v^{(i,j)_b}$ of $W$, which arise by relaxing the two facets of $W$ associated with the same edge: $(i,j)$.  In case 1, $(i,j) \in A$, and in case 2,  $(i,j) \in \bar{A}$.

For odd index $r \in \{3, 5, \dots, 2m-1\}$, let $i = (r+1)/2$; the calculation of $\tilde{M}_r$ involves two vertices, $v^{(h,i)_b}$ and $v^{(i,j)_a}$, that arise by relaxing facets associated with adjacent edges $(h,i)$ and $(i,j)$ of $C_m$.  Cases 3, 4, 5, and 6 cover all possible combinations of inclusion/exclusion of $(h,i)$ and $(i,j)$ in the set A. We can now specify, for $r \in \{3, 5, \dots, 2m-1\}$,
\[
\tilde{M}_{r} :=
\left\{
\begin{array}{rl}
M_r - M_{r-1}, & \text{ if } (h,i) \in A; \\
M_r + M_{r-1}, & \text{ if } (h,i) \in \bar{A}. \\
\end{array}
\right.
\]

\bigskip
\noindent\textbf{Case 1:} $r = 2i \in \{2,4,\dots,2m\};~ (i,j) \in A.$

The columns of $\tilde{M}$ correspond to the coordinates of $v = (v_1, v_2, v_3, \dots, v_{2m}) = (x_1, y_{(1,2)}, x_2, \dots, y_{(m,1)})$, so that for $r=2i$, columns $r-1$, $r$, and $r+1$ correspond to variables $x_i$, $y_{(i,j)}$, and $x_j$, respectively.
Referring to Table \ref{tab:ij_values}, we have the following values for calculating coordinates $r-1$, $r$, and $r+1$ of $\tilde{M}_r = (v^{(i,j)_b} + v^{(i,j)_a} - 2v^0)'$:
\begin{center}
\begin{tabular}{l|c|c|c}
& $x^*_i$ & $y^*_{(i,j)}$ & $x^*_j$ \\
\hline
$v^0$ & 1/2 & 0 & 1/2 \\
\hline
$v^{(i,j)_a}$ & 1 & 1 & 1 \\
\hline
$v^{(i,j)_b}$ & 0 & 0 & 0 \\
\end{tabular}~;
\end{center}
so in this case we have
\[
\begin{array}{lcccr}
\tilde{M}_{r,r-1} &=& x^{(i,j)_b}_i + x^{(i,j)_a}_i - 2x^0_i &=& 0,\\
\tilde{M}_{r,r} &=& y^{(i,j)_b}_{(i,j)} + y^{(i,j)_a}_{(i,j)} - 2y^0_{(i,j)} &=& 1,\\
\tilde{M}_{r,r+1} &=& x^{(i,j)_b}_j + x^{(i,j)_a}_j - 2x^0_j &=& 0.
\end{array}
\]

Now we consider the remaining odd-index coordinates of $\tilde{M}_r$, which correspond to all remaining $x$-variables (besides $x_i$ and $x_j$).  Consider traversing the coordinates of $v^{(i,j)_a}$ and $v^{(i,j)_b}$ to the right (in the direction of higher indices), starting from the $x_j$ coordinate of both.  Referring to Table \ref{tab:kl_values}, if the next edge $(j,\hat{j})$ (adjacent to $(i,j)$) is in $A$, then
\[
(x^{(i,j)_a}_j, y^{(i,j)_a}_{(j,\hat{j})},x^{(i,j)_a}_{\hat{j}}) = (1,0,0)
\]
and
\[
(x^{(i,j)_b}_j, y^{(i,j)_b}_{(j,\hat{j})},x^{(i,j)_b}_{\hat{j}}) = (0,0,1).
\]
If $(j,\hat{j}) \in \bar{A}$, then
\[
(x^{(i,j)_a}_j, y^{(i,j)_a}_{(j,\hat{j})},x^{(i,j)_a}_{\hat{j}}) = (1,1,1)
\]
and
\[
(x^{(i,j)_b}_j, y^{(i,j)_b}_{(j,\hat{j})},x^{(i,j)_b}_{\hat{j}}) = (0,0,0).
\]
Notice that whether or not $(j,\hat{j})$ is in $A$, $x^{(i,j)_a}_{\hat{j}}$ and $x^{(i,j)_b}_{\hat{j}}$ have opposite values, as was initiated at $x_j$.
In fact, due to the possible coordinates of each triple $(x^{(i,j)_*}_k, y^{(i,j)_*}_{(k,\ell)},x^{(i,j)_*}_{\ell})$ (see Table \ref{tab:kl_values}), as we traverse the coordinates of $v^{(i,j)_a}$ and $v^{(i,j)_b}$ from left to right starting at $x_j$, including wrapping around to $x_1$ and continuing to $x_i$ (because this is a cycle), we see that, for all $k \in V$, $x^{(i,j)_a}_{k}$ and  $x^{(i,j)_b}_{k}$ maintain an opposite pattern throughout the vector, so that $x^{(i,j)_a}_{k} + x^{(i,j)_b}_{k} = 1$, and therefore $\tilde{M}_{r, 2k-1} = x^{(i,j)_a}_{k} + x^{(i,j)_b}_{k} - 2x^0_{k} = 0$.

We complete the computation of $\tilde{M}_{r}$ by considering the rest of its even-index coordinates, which correspond to $y$-variables. Again, we refer to Table \ref{tab:kl_values}.  Let $(k,\ell) \in E\backslash \{(i,j)\}$.  If $(k,\ell) \in A$, then $y^{(i,j)_a}_{(k,\ell)} = y^{(i,j)_b}_{(k,\ell)} = y^0_{(k,\ell)} = 0$.  If $(k,\ell) \in \bar{A}$, $y^{(i,j)_a}_{(k,\ell)}$ and  $y^{(i,j)_b}_{(k,\ell)}$ have opposite values (as described above for $x^{(i,j)_a}_{\ell}$ and  $x^{(i,j)_b}_{\ell}$ ), and $y^0_{(k,\ell)} = 1/2$.  In either case, $\tilde{M}_{r,2k} = y^{(i,j)_a}_{(k,\ell)} + y^{(i,j)_b}_{(k,\ell)} - 2y^0_{(k,\ell)} = 0$.

In summary,
\[
\tilde{M}_{r,k} =
\left\{
\begin{array}{rl}
1, & \text{ if } k = r, \\
0, & \text{ otherwise}.
\end{array}
\right.
\]

\smallskip
\noindent\textbf{Case 2:} $r = 2i \in \{2,4,\dots,2m\};~ (i,j) \in \bar{A}.$

Again we have $\tilde{M} = (v^{(i,j)_b} + v^{(i,j)_a} - 2v^0)'$.   According to Table \ref{tab:ij_values}, we have for indices $r-1$, $r$, and $r+1$,
\begin{center}
\begin{tabular}{l|c|c|c}
& $x^*_i$ & $y^*_{(i,j)}$ & $x^*_j$ \\
\hline
$v^0$ & 1/2 & 1/2 & 1/2 \\
\hline
$v^{(i,j)_a}$ & 1 & 0 & 0 \\
\hline
$v^{(i,j)_b}$ & 0 & 0 & 1 \\
\end{tabular}~,
\end{center}
so that
\[
\begin{array}{lcccr}
\tilde{M}_{r,r-1} &=& x^{(i,j)_b}_i + x^{(i,j)_a}_i - 2x^0_i  &=& 0,\\
\tilde{M}_{r,r} &=&  y^{(i,j)_b}_{(i,j)} + y^{(i,j)_a}_{(i,j)} - 2y^0_{(i,j)}  &=& -1,\\
\tilde{M}_{r,r+1} &=& x^{(i,j)_b}_j + x^{(i,j)_a}_j - 2x^0_j  &=& 0.
\end{array}
\]
As in Case 1 (and by a similar argument), the remainder of the coordinates of $\tilde{M}_r$ are all zero.  In summary,
\[
\tilde{M}_{r,k} =
\left\{
\begin{array}{rl}
-1 & \text{ if } k = r; \\
0 & \text{ otherwise}.
\end{array}
\right.
\]

\smallskip
\noindent\textbf{Case 3:} $r = 2i-1 \in \{3,5,\dots,2m-1\}; ~ (h,i),~(i,j) \in A.$

In this case, we take $\tilde{M}_r := M_r - M_{r-1} = (v^{(i,j)_a} - v^{(h,i)_b})'$.  Extreme points $v^{(h,i)_b}$ and $v^{(i,j)_a}$ arise by relaxing facets associated with two adjacent edges, $(h,i)$ and $(i,j)$, respectively.  In all, five variables appear in these two facets:  $x_h$, $y_{(h,i)}$, $x_i$, $y_{(i,j)}$, and $x_j$, corresponding to column indices $r-2$, $r-1$, $r$, $r+1$, and $r+2$, respectively.  Referring to both Tables \ref{tab:ij_values} and \ref{tab:kl_values} as appropriate, we obtain the values,

\begin{center}
\begin{tabular}{l|c|c|c|c|c}
 & $x^*_{h}$ & $y^*_{(h,i)}$ & $x^*_i$ & $y^*_{(i,j)}$ & $x^*_j$ \\
\hline
$v^{(h,i)_b}$ & 0 & 0 & 0 & 0 & 1 \\
\hline
$v^{(i,j)_a}$ & 0 & 0 & 1 & 1 & 1 \\
\end{tabular}~,
\end{center}
and we can now calculate,
\[
\begin{array}{lcccr}
\tilde{M}_{r,r-1} &=&  y^{(i,j)_a}_{(h,i)} - y^{(h,i)_b}_{(h,i)}  &=& 0,\\
\tilde{M}_{r,r} &=& x^{(i,j)_a}_i - x^{(h,i)_b}_i &=& 1,\\
\tilde{M}_{r,r+1} &=& y^{(i,j)_a}_{(i,j)} - y^{(h,i)_b}_{(i,j)}  &=& 1.
\end{array}
\]

Unlike in cases 1 and 2, the $x$-coordinates in the last column of the table above match in $v^{(h,i)_b}$ and $v^{(i,j)_a}$; i.e.,  $x^{(h,i)_b}_k = x^{(i,j)_a}_k$, for $k \in \{h, j\}$.  As described in case 1, the $x$-coordinates maintain the same (matching, in this case) pattern throughout (except at $x_i$).  So for all $k \in V\backslash \{i\}$, $x^{(h,i)_b}_k = x^{(i,j)_a}_k$.  Therefore, $x^{(i,j)_a}_k - x^{(h,i)_b}_k = 0$, for $k \in V\backslash \{i\}$.

For the $y$-coordinates of $\tilde{M}_r$ (besides those in the table above), we refer to Table \ref{tab:kl_values}.  Noting that $v^{(h,i)_b}$ and $v^{(i,j)_a}$ match on all $x$-coordinates, they must also match at the rest of the $y$-coordinates.  For example, if $(k,\ell) \in A\backslash \{(h,i),(i,j)\}$, and $x^{(h,i)_b}_k = x^{(i,j)_a}_k = 1$, then by Table \ref{tab:kl_values}, $y^{(h,i)_b}_{(k,\ell)} = y^{(i,j)_a}_{(k,\ell)} = 0$ (and $x^{(h,i)_b}_{\ell} = x^{(i,j)_a}_{\ell} = 0$).  That is, there is only one choice for the threesome $(x^*_k,~ y^*_{(k,\ell)},~ x^*_{\ell})$  in Table \ref{tab:kl_values}, once the $(k,\ell)$ is fixed with respect to belonging to $A$ or $\bar{A}$, and $x^*_k$ is fixed to either 0 or 1, for $* \in \{(h,i)_b,(i,j)_a\}$.  We can conclude that $y^{(i,j)_a}_{(k,\ell)} - y^{(h,i)_b}_{(k,\ell)} = 0$, for $(k,\ell) \in E\backslash \{(h,i),(i,j)\}$.

In summary,
\[
\tilde{M}_{r,k} =
\left\{
\begin{array}{rl}
1, & \text{ if } k \in \{r, r+1\}, \\
0, & \text{ otherwise}.
\end{array}
\right.
\]

\smallskip
\noindent\textbf{Case 4:} $r = 2i-1 \in \{3,5,\dots,2m-1\}; ~ (h,i),~(i,j) \in \bar{A}.$

In this case, we take $\tilde{M}_r = M_r + M_{r-1} = (v^{(i,j)_a} + v^{(h,i)_b} - 2v^0)'$.  As in case 3, the two involved extreme points arise by relaxing facets corresponding to two adjacent edges, and so involve five adjacent variables, with $x_i$ corresponding to column index $r$.  We obtain the following table of values:
\begin{center}
\begin{tabular}{l|c|c|c|c|c}
 & $x^*_{h}$ & $y^*_{(h,i)}$ & $x^*_i$ & $y^*_{(i,j)}$ & $x^*_j$ \\
\hline
$v^0$ & 1/2 & 1/2 & 1/2 & 1/2 & 1/2 \\
\hline
$v^{(h,i)_b}$ & 0 & 0 & 1 & 1 & 1 \\
\hline
$v^{(i,j)_a}$ & 1 & 1 & 1 & 0 & 0 \\
\end{tabular}~,
\end{center}
leading to
\[
\begin{array}{lcccr}
\tilde{M}_{r,r-1} &=&  y^{(i,j)_a}_{(h,i)} + y^{(h,i)_b}_{(h,i)} - 2y^0_{(h,i)}  &=& 0,\\
\tilde{M}_{r,r} &=& x^{(i,j)_a}_i + x^{(h,i)_b}_i - 2x^0_{i} &=& 1,\\
\tilde{M}_{r,r+1} &=& y^{(i,j)_a}_{(i,j)} + y^{(h,i)_b}_{(i,j)} - 2y^0_{(i,j)}   &=& 0.
\end{array}
\]
Following similar arguments as those in the previous cases, all other coordinates of $M_r$ are zero, so that
\[
\tilde{M}_{r,k} =
\left\{
\begin{array}{rl}
1, & \text{ if } k = r, \\
0, & \text{ otherwise}.
\end{array}
\right.
\]

\smallskip
\noindent\textbf{Case 5:} $r = 2i-1 \in \{3,5,\dots,2m-1\}; ~ (h,i) \in A, ~(i,j) \in \bar{A}.$

As in case 3, we take $\tilde{M}_r := M_r - M_{r-1} = (v^{(i,j)_a} - v^{(h,i)_b})'$, and again the $x$-coordinates in the outside columns ($x_h$ and $x_j$) match:
\begin{center}
\begin{tabular}{l|c|c|c|c|c}
 & $x^*_{h}$ & $y^*_{(h,i)}$ & $x^*_i$ & $y^*_{(i,j)}$ & $x^*_j$ \\
\hline
$v^{(h,i)_b}$ & 0 & 0 & 0 & 0 & 0 \\
\hline
$v^{(i,j)_a}$ & 0 & 0 & 1 & 0 & 0 \\
\end{tabular}~,
\end{center}
so $v^{(h,i)_b}$ and $v^{(i,j)_a}$ agree in all coordinates except in $x_i$ (index $r$):
\[
\tilde{M}_{r,k} =
\left\{
\begin{array}{rl}
1, & \text{ if } k = r, \\
0, & \text{ otherwise}.
\end{array}
\right.
\]

\smallskip
\noindent\textbf{Case 6:}  $r = 2i-1 \in \{3,5,\dots,2m-1\}; ~ (h,i) \in \bar{A},
~(i,j) \in A.$

In this last case, we set $\tilde{M}_r = M_r + M_{r-1} = (v^{(i,j)_a} + v^{(h,i)_b} - 2v^0)'$, and we have
\begin{center}
\begin{tabular}{l|c|c|c|c|c}
 & $x^*_{h}$ & $y^*_{(h,i)}$ & $x^*_i$ & $y^*_{(i,j)}$ & $x^*_j$ \\
\hline
$v^0$ & 1/2 & 1/2 & 1/2 & 0 & 1/2 \\
\hline
$v^{(h,i)_b}$ & 0 & 0 & 1 & 0 & 0 \\
\hline
$v^{(i,j)_a}$ & 1 & 1 & 1 & 1 & 1 \\
\end{tabular}~.
\end{center}
We omit the details, which are similar to previous cases, but the result is
\[
\tilde{M}_{r,k} =
\left\{
\begin{array}{rl}
1, & \text{ if } k \in \{r, r+1\}, \\
0, & \text{ otherwise}.
\end{array}
\right.
\]

\bigskip
This concludes our proof that $\tilde{M}$ is upper triangular with diagonal coordinates
\[
|\tilde{M}_{r,r}| =
\left\{
\begin{array}{rl}
1/2, & \text{ if } r = 1, \\
1, & \text{ otherwise},
\end{array}
\right.
\]
and the result follows.
\end{proof}

Combining Theorem \ref{vol_QCm} with Lemmas \ref{lem:disj} and \ref{cutoffhalf}, we obtain the following result --- a closed-form expression for the volume of the boolean quadric polytope of a cycle.

\begin{thm} \label{volPCm}
For $m \geq 3$,
\[
\vol(\scrP(C_m)) ~=~ \vol(\scrQ(C_m))-\frac{2^{m-1}}{2(2m)!} ~=~ \frac{m A_{2m-1}}{2^m (2m)!}-\frac{2^{m-2}}{(2m)!}.
\]
\end{thm}

\section{Asymptotics} \label{sec:asymp}

In this section, we make some asymptotic analyses. In doing so, we hope to
learn something about the behavior of the polytopes that we are studying, as problem sizes grows.
This type of analysis follows the spirit of many of the results in \cite{LM1994}.
Furthermore, it is not simple to compare the combinatorial formulae that we have derived,
and so asymptotics provides a clear lens for comparing.
After invoking Stirling's formula or  Andr\'e's Theorem as appropriate,
our asymptotic results can be checked with \verb;Mathematica;.

We demonstrate for a wide variety of graphs,
in particular: complete graphs (Proposition \ref{prop:asymp_QKn}), stars and paths
(Corollary \ref{cor:asympSP}),
and cycles (first half of Corollary \ref{cor:asympCm}),
that $\scrQ$ is a large portion of the unit hypercube.
This is important for putting into context all
current and future results on volumes of
boolean quadric polytopes and their relaxations.

A highlight of this section is a demonstration that the volume of $\scrQ$
outside of $\scrP$ is quite small, compared to  the volume of $\scrQ$,
when $G$ is a single long cycle  (Corollary \ref{cor:asympCm}); this is despite the
fact that when $G$ is a single long cycle,
the description of $\scrP$ is much heavier than that of $\scrQ$.
In sharp contrast to this, we demonstrate that when $G$ is a collection of
many small cycles (in fact, triangles),
then the  volume of $\scrQ$
outside of $\scrP$ is quite large, compared to  that of $\scrQ$
(Corollary \ref{cor:Q3}), while the
description of $\scrP$ is \emph{not} much heavier than that of $\scrQ$.
In particular, we see that for graphs having each edge in exactly one
cycle, we can get very different behaviors. Graphs with each edge in at most one
cycle are
cactus forests. So, as discussed more in \S\ref{sec:conclusions},
it becomes interesting to try and get an efficient algorithm for calculating
the volume of $\scrP$ for a cactus forest; note that non-trivial cactus forests
have treewidth 2, so we already have an  efficient algorithm for calculating
the volume of $\scrQ$.

As indicated in \cite{LM1994}, it is natural to compare
sets in $\mathbb{R}^d$ by comparing the $d$-th roots of their
volumes. Because we take $d$-th roots, we have to be precise about the
ambient dimension $d$ for our polytopes.
So in what follows we assume that our graphs have no isolated vertices, and we always regard our polytopes
as being in dimension $d=n+m$, the least dimension that
makes sense (rather than in dimension $n+\binom{n}{2}$).

As we have mentioned at the outset,  \cite{KLS1997}
 established the $d$-dimensional volume of $\scrQ(K_n)$ to be $2^{2n-d}n!/(2n)!$,
 where $d=n+\binom{n}{2}$.   Invoking Stirling's formula, it is easy to check
 the following calculation.

 \begin{prop} \label{prop:asymp_QKn}
\begin{equation} \label{asymp_QKn}
\lim_{n \to \infty} \vol_{d}(\scrQ(K_n))^{1/d} ~=~ \frac{1}{2},
\end{equation}
where $d=n+\binom{n}{2}$.
\end{prop}
This is quite substantial, as the volume of the entire unit hypercube and its $d$-th root are
of course unity. It is an outstanding open problem, first posed in  \cite{KLS1997} and which we would like to highlight,
to understand how close $\vol_{d}(\scrQ(K_n))$ and $\vol_{d}(\scrP(K_n))$
are, asymptotically.

When $G$ is a forest, $\scrP$ and $\scrQ$ are the same.  Still it is interesting to compare the asymptotics of $\scrQ(G)$ and $\scrQ(G')$ when connected $G$ and $G'$ have the same number of edges.
Our next result does this for two very different trees on $m$ edges.

\begin{cor} \label{cor:asympSP}
\begin{equation} \label{asymp_QSm}
\lim_{m \to \infty} \vol_{2m+1}(\scrQ(S_m))^{1/(2m+1)} ~=~ \frac{1}{2}
\end{equation}

 and

 \begin{equation} \label{asymp_QPm}
\lim_{m \to \infty} \vol_{2m+1}(\scrQ(P_m))^{1/(2m+1)} ~=~ \frac{\sqrt{2}}{\pi}
~\approx~ 0.450158.
 \end{equation}
\end{cor}

\begin{proof}
It is easy to check \eqref{asymp_QSm} using Stirling's formula.

By Theorem \ref{thm:pathvol}, we have that
$\vol_{2m+1}(\scrQ(P_m))  = \frac{A_{2m+1}}{2^m(2m+1)!}$.
By Andr\'e's Theorem \ref{thm:andre}, we have
\[
A_k /k! = \frac{4}{\pi} \left(\frac{2}{\pi}\right)^k +
\mathcal{O}\left(
\left(\frac{2}{3\pi}\right)^k
\right)
\]
(see \cite{Stanley_Alt_Survey}).  Combining these facts,  \eqref{asymp_QPm} follows easily.
\end{proof}

It is interesting to observe that the path and star, both on $m$ edges and $m+1$ vertices,
and hence having their associated polytopes naturally living in dimension $2m+1$,
behave substantially similarly though non-trivially differently, from our point of view.

Next, we demonstrate that for $C_m$,
the volume of $\scrQ$ is quite large
compared to the volume of $\scrQ$ that is
outside of $\scrP$.
So, we have a family of examples
demonstrating that $\scrP$ can have a description
that has many more inequalities than $\scrQ$,
while their volumes are very close.  In particular,
when $G$ is a cycle, $\scrQ$ has only $4m$ facets, while
$\scrP$ has $4m + 2^{m-1}$ facets.

\begin{cor} \label{cor:asympCm}

\begin{equation} \label{asymp_QCm}
\lim_{m \to \infty} \vol_{2m}(\scrQ(C_m))^{1/2m} ~=~ \frac{\sqrt{2}}{\pi} ~\approx~ 0.450158
\end{equation}

 and

 \begin{equation} \label{asymp_QCm_diff}
 \lim_{m \to \infty}
 m \times \vol_{2m}\left(\scrQ(C_m) \setminus \scrP(C_m)\right)^{1/2m} ~=~
 \frac{e}{\sqrt{2}}.
 \end{equation}
\end{cor}

\begin{proof}
By Theorem \ref{vol_QCm}, we have that $\vol_{2m}(\scrQ(C_m)) = \frac{A_{2m-1}}{2^{m+1}(2m-1)!}$.  Using
again Andr\'e's Theorem and Stirling's formula (as in the proof of  Corollary \ref{cor:asympSP}),  \eqref{asymp_QCm} follows easily.

Because $\scrP(C_m) \subseteq \scrQ(C_m)$,
Theorem \ref{volPCm}
implies that
\[
\vol_{2m}\left(\scrQ(C_m) \setminus \scrP(C_m)\right) = 2^{m-2}/(2m)!.
\]
  Invoking  Stirling's formula, we have
\[
\vol_{2m}(\scrQ(C_m) \setminus \scrP(C_m))^{1/2m} \sim \left(\frac{2^{m-2}}{\sqrt{4m \pi}\left(\frac{2m}{e}\right)^{2m}}\right)^{1/2m} \sim \frac{e}{m\sqrt{2}},
\]
and \eqref{asymp_QCm_diff} follows.
\end{proof}

We note that because $P_m$ is a forest, $\scrP(P_m)=\scrQ(P_m)$.
While of course $C_m$ is not a forest, and so  $\scrP(C_m)\not=\scrQ(C_m)$.
One way this different behavior manifests itself is in the explosion
of the number of facets for $\scrP(C_m)$.
But in some sense the graphs $P_m$ and $C_m$ do not look very different,
and we see this echoed in the facts that: (i) the asymptotic behavior of
their volumes is identical (compare \eqref{asymp_QPm} and \eqref{asymp_QCm}),
and (ii) $\scrQ(P_m) \setminus \scrP(P_m) =\emptyset$ while $\vol_{2m}(\scrQ(C_m) \setminus \scrP(C_m))^{1/2m}$
decays like $e/m\sqrt{2}$, so it is nearly zero.
\\

From what we have seen so far, it is natural to wonder whether $\scrP(G)$ and $\scrQ(G)$ are always close, when $G$ is sparse.
For a natural number $p$, let $C^p_3$ denote
a graph that is the disjoint union of $p$ copies of the triangle
$C_3$. For $m$ divisible by 3, we wish to compare the behaviors of  $\scrQ(C^{m/3}_3)$ and
$\scrQ(C^{m/3}_3)\setminus \scrP(C^{m/3}_3)$ with
those of
$\scrQ(C_m)$ and
$\scrQ(C_m)\setminus\scrP(C_m)$.

\begin{cor}\label{cor:Q3}

\begin{equation} \label{Q3}
 \vol_{2m}(\scrQ(C^{m/3}_3))^{1/2m} ~=~ \left(\frac{1}{120}\right)^{1/6} ~\approx~ 0.450267
\end{equation}

 and

 \begin{equation} \label{Q3_diff}
\vol_{2m}(\scrQ(C^{m/3}_3)\setminus\scrP(C^{m/3}_3))^{1/2m} ~=~ \left(\frac{1}{360}\right)^{1/6}
 ~\approx~ 0.374929.
 \end{equation}
\end{cor}

\begin{proof}
Easily follows from Theorem \ref{vol_QCm}, Theorem \ref{volPCm} and Lemma \ref{crossproduct}.
\end{proof}

It is very interesting to see, comparing \eqref{asymp_QCm} with  \eqref{Q3},
 that $\vol_{2m}(\scrQ(C_m))$
has a remarkably similar behavior to $\vol_{2m}(\scrQ(C^{m/3}_3))$,
while, comparing  \eqref{asymp_QCm_diff} with \eqref{asymp_QCm} and \eqref{Q3_diff},
 $\vol_{2m}\left(\scrQ(C_m) \setminus \scrP(C_m)\right)$
is quite small compared to both $\vol_{2m}(\scrQ(C_m))$ and to
$\vol_{2m}(\scrQ(C^{m/3}_3)\setminus\scrP(C^{m/3}_3))$.
That is, for one long cycle $C_m$, the polytope $\scrP$,
which needs  $4m+2^{m-1}$ facets to describe\footnote{for $\scrP(C_m)$,
there are $4m$ facets corresponding to inequalities of the type
 \ref{facet0}-\ref{facet3} (4 per edge),
 and $2^{m-1}$ facets corresponding to inequalities of the type \ref{oddcyclefacet}
 ($2^{m-1}$ odd subsets of an $m$-element set); see \cite{Padberg1989}}, is well approximated
by the  polytope $\scrQ$, which needs only $4m$
facets to describe. On the other hand,
for a collection of $m/3$ triangles (which has the same number of
edges as $C_m$), the
polytope $\scrP$ only needs $16m/3$ facets to describe
and in this case it is not well approximated by
$\scrQ$, which still needs only $4m$
facets to describe.

The takeaway from our comparative asymptotic volume analysis, is that for one long cycle,
we get a good approximation of the heavy $\scrP$ using a relatively very-light
relaxation $\scrQ$. While for a collection of triangles, we gain a lot using what turns out to be a not-very-heavy refinement of $\scrQ$.
We can reasonably hope that this careful analysis
gives us a useful message for other sparse graphs, and in pursuit of that,
in the next section, we make some computational experiments.

\section{Computational experiments} \label{sec:exper}

In this section, we report on computational experiments designed to see what behavior of cycles, seen in \S\ref{sec:asymp}, persists for a family of more complex graphs. In \S\ref{sec:asymp}, we have seen different behaviors for cycles, depending on their lengths.  So, we consider a class of graphs that has attributes of both a long cycle and many small cycles.
For $n\geq 3$, we define the \emph{$n$-necklace} $N_n$ to be a cycle $C_n$ with a triangle, i.e. $C_3$, ``hanging'' from each vertex of $C_n$. In Figure \ref{8necklace}, we depict $N_8$.

\begin{figure}
\begin{center}

	\includegraphics[width=.45\linewidth]{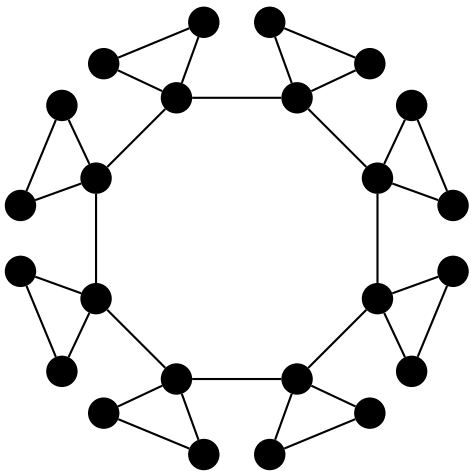}
\end{center}
\caption{8-necklace}  \label{8necklace}
\end{figure}

Our goal is to compare the volumes of various relaxations
of $\scrP(N_n)$. We compare them via their $d$-th roots,
following the spirit of \S\ref{sec:asymp}.
The polytope $\scrR(N_n)$ is the part of the basic relaxation $\scrQ(N_n)$
that satisfies the $2^{n-1}$ odd cycle
inequalities associated with the big cycle $C_n$.
The polytope $\scrT(N_n)$ is the part of the basic relaxation $\scrQ(N_n)$
that satisfies the four odd cycle inequalities
associated with each triangle $C_3$.
Finally, we have the usual boolean quadric polytope $\scrP(N_n)$ of $N_n$,
which we note is the intersection of $\scrR(N_n)$ and $\scrT(N_n)$
--- this is because $N_n$ is a series-parallel graph, and for
series-parallel graphs, $\scrP(N_n)$ is completely described
 by \ref{facet0}-\ref{oddcyclefacet} (see \cite[Theorem 10]{Padberg1989}).

Using {\tt LEcount} (see \cite{LeCount} and \cite{Kangas2016}),
we exactly calculated $\left(\vol_{d}(\scrQ(N_n))\right)^{1/d}$, for
$n=4,\ldots,10$.  These numbers appear in the second column of Table \ref{VN}.
Note that the polytope $\scrQ(N_n)$ lives in dimension
$d=7n$.
We stopped after $n=10$, due to memory issues.
But we can easily observe that to 6 decimal places,
we have a clear picture of the limiting constant.

For $n=4,\ldots,12$, we approximated the volumes of
several related polytopes, using the {\tt Matlab}
software \cite{MatlabVol}. These numbers appear
in the remaining columns of Table \ref{VN}.
Note that in the
{\tt Matlab} software, we varied the precision $\eps$ depending on the
dimension, so as to approximate the $d$-th roots to a
precision of roughly $\delta=0.0001$. So we set $\eps:=
(1+\delta)^d-1$.

Generally, $\scrT(N_n)$ is a  very light refinement of
$\scrQ(N_n)$, while $\scrR(N_n)$ is a rather heavy refinement.
For example, for $n=12$, as compared to $\scrQ(N_{12})$,
$\scrR(N_{12})$ uses 2048 extra inequalities, while
$\scrT(N_{12})$ uses only 48 extra inequalities, and hence
$\scrP(N_{12})$ uses  2096 extra inequalities.
What we can easily see is that we almost completely capture
$\scrP(N_n)$ with the very light  $\scrT(N_n)$.
Furthermore, the very heavy $\scrR(N_n)$ leaves a significant
gap to $\scrP(N_n)$. In summary, the messages of Corollaries \ref{cor:asympCm} and \ref{cor:Q3} extend to more complicated
situations: \emph{odd cycle inequalities seem to be more
important for short cycles  than  long cycles.} We note that this
echoes the message of \cite{SparseMolinaro}, where sparse inequalities are
shown to have more power than dense ones (in certain structured settings).



\begin{table}
\[
\begin{array}{ccccc}
 \hline
 n & \scrQ & \mathscr{R} & \mathscr{T} & \scrP \\
 \hline
 4 & 0.573963 & 0.425950 & 0.396662 & 0.399680 \\
 5 & 0.573963 & 0.426805 & 0.400130 & 0.399400 \\
 6 & 0.573963 & 0.426061 & 0.399665 & 0.399436 \\
 7 & 0.573963 & 0.428294 & 0.400695 & 0.400313 \\
 8 & 0.573963 & 0.426034 & 0.399421 & 0.400937 \\
 9 & 0.573963 & 0.425517 & 0.399619 & 0.400723 \\
 10 & 0.573963 & 0.426842 & 0.401514 & 0.400903 \\
 11 & $*$ & 0.426218 & 0.400749 & 0.400700 \\
 12 & $*$ & 0.426597 & 0.400482 & 0.400965
\end{array}
\]
\caption{Comparison of relaxations of the BQP for $n$-necklaces}\label{VN}
\end{table}

\section{Future work} \label{sec:conclusions}

A very challenging open problem is to get a polynomial-time algorithm
for calculating $\vol(\scrP(G))$ when $G$ is a series-parallel graph (briefly, the class of graphs having no $K_4$ graph minor). As we have mentioned, in this case,
$\scrP(G)$ is completely described by  \ref{facet0}-\ref{oddcyclefacet}.
Even for the subclass of outerplanar graphs (briefly, the class of graphs having no $K_4$ nor $K_{2,3}$ graph minor), this is already very challenging because
the number of cycles in such a graph can be exponential (in fact $>\Omega(1.5^n)$, see \cite{deMier2012}). A more manageable problem would be to restrict
our attention to the further subclass of cactus forests, i.e. graphs
where each edge is in no more than one cycle --- they can also be understood as the class of graphs having no diamond (i.e., $K_4$ minus an edge)
graph minor. The necklaces $N_n$ (see \S\ref{sec:exper}) are cactus graphs.
 Cactus graphs occur in a wide variety of applications,
e.g., location theory, communication networks, and stability analysis
(see \cite{Novak} and the references therein).
Cactus graphs can be recognized in linear time, via a depth-first search approach (see \cite{Novak} and \cite{Zmazek2003}), and of course the number of cycles in
such a graph is at most $n/3$. We do note that we can apply
Lemma \ref{lem:disj}, which tells us that odd cycle inequalities from the
same odd cycle cut off disjoint parts of $\scrQ(G)$.
But it already seems to be quite a
challenging problem to characterize the volume of $\scrQ(G)$
cut off by a single odd cycle inequality (seeking to generalize
Lemma \ref{cutoffhalf}).

For a class of inequalities and a given point, the associated \emph{separation problem} is to determine (if
there is) an inequality in the class that is not satisfied by the point. It is very well known (see \cite{GLS})
that efficiently being able to solve the separation problem for the inequalities that describe a polytope
is equivalent to efficiently being able to optimize a linear function on that polytope.
It is interesting to note that while odd cycle inequalities can be separated in $\mathcal{O}(n^3)$ time (see \cite{Barahona1989}),
we can actually devise a very simple $\mathcal{O}(n^2)$ approach for
cactus graphs. First, we enumerate the simple cycles of the cactus graph.
Then for each simple cycle, we merely have to check, for a given fixed $(x,y)$ if
there is an $A$ for which  \ref{oddcyclefacet} is violated.
The simplest way to see how to proceed involves using the affine equivalence of $\scrP(K_n)$ with the cut polytope of $K_{n+1}$ (see \cite{BQ6}).
For the cut polytope (which has variables indexed only by edges), the odd cycle inequalities can be written in the form:
\begin{equation}\label{oddcycle_cutpolytope}
\sum_{e\in F} (1-z_e) + \sum_{e\in C\setminus F} z_e \geq 1,
\end{equation}
where $C$ is a cycle and $F\subset C$ has odd cardinality.
Using an idea in \cite[proof of Proposition 1.7]{CopperLee}, for a point $\bar{z}$, we can let
\[
F':= \left\{ e\in C ~:~ \bar{z}_e > 1/2\right\}.
\]
Then $\bar{z}$ can violate \eqref{oddcycle_cutpolytope}
only if $|F\Delta F'|\leq 1$, where $F\Delta F':=(F\cup F')\setminus (F\cap F')$ (that is, the usual
\emph{symmetric difference} of $F$ and $F'$). So either $F'$ has odd cardinality and
we need only check \eqref{oddcycle_cutpolytope} for $F=F'$,
or $F$ has even cardinality, and we need only check \eqref{oddcycle_cutpolytope} for the at most
$|C|$ (odd) sets $F$ satisfying  $|F\Delta F'|= 1$.

\section*{Acknowledgments}
J. Lee was partially supported by ONR grants N00014-14-1-0315 and N00014-17-1-2296.
The authors gratefully acknowledge Komei Fukuda (developer of {\tt cdd})
and Benno B\"ueler and Andreas Enge (developers of {\tt VINCI});
use of their software was invaluable for formulating conjectures that became theorems.


\section*{References}

\bibliography{volume}

\end{document}